\documentclass[11pt]{amsart}
\usepackage{amscd}
\usepackage{amsmath}
\usepackage{amssymb}
\usepackage{amsthm}
\newtheorem{theorem}{Theorem}[section]
\newtheorem{lemma}[theorem]{Lemma}

\newtheorem{conjecture}[theorem]{Conjecture}

\theoremstyle{definition}

\theoremstyle{remark}

\numberwithin{equation}{section}

\DeclareMathOperator{\pd}{pd}

\DeclareMathOperator{\depth}{depth}

\DeclareMathOperator{\K-d}{Krull-dim}

\DeclareMathOperator{\lcm}{lcm}

\DeclareMathOperator{\lex}{lex}
\DeclareMathOperator{\initial}{in}

\def\tb{{\bold t}}

\def\xb{{\bold x}}

\begin{document}
\title[Depth of edge rings]
{Depth of edge rings arising from finite graphs}
\author[]{Takayuki Hibi}
\address[]{Department of Pure and Applied Mathematics,
Graduate School of Information Science and Technology,
Osaka University,
Toyonaka, Osaka 560-0043, Japan}
\email{hibi@math.sci.osaka-u.ac.jp}
\author[]{Akihiro Higashitani}
\address[]{Department of Pure and Applied Mathematics,
Graduate School of Information Science and Technology,
Osaka University,
Toyonaka, Osaka 560-0043, Japan}
\email{sm5037ha@ecs.cmc.osaka-u.ac.jp}
\author[]{Kyouko Kimura}
\address[]{Department of Pure and Applied Mathematics,
Graduate School of Information Science and Technology,
Osaka University,
Toyonaka, Osaka 560-0043, Japan}
\email{kimura@math.sci.osaka-u.ac.jp}
\author[]{Augustine B. O'Keefe}
\address[]{Mathematics Department,
Tulane University,
6823 St.\  Charles Ave,
New Orleans, LA 70118, U.S.A}
\email{aokeefe@tulane.edu}
\thanks{
The fourth author had summer support provided by the JSPS Research
Fellowships for Young Scientists and the NSF East Asia and Pacific Institutes Fellowship. 
}
\begin{abstract}
Let $G$ be a finite graph and $K[G]$ the edge ring of $G$.
Based on the technique of Gr\"obner bases and initial ideals,
it will be proved that,
given integers $f$ and $d$ with $7 \leq f \leq d$,
there exists a finite graph $G$ on $[d] = \{ 1, \ldots, d \}$ with
$\depth K[G] = f$ and with $\K-d K[G] = d$. 
\end{abstract}
\maketitle

\section*{Introduction}
The edge ring \cite{OH98}
and its toric ideal \cite{OH99} arising from a finite graph 
have been studied from viewpoints of both commutative algebra
and combinatorics.  Especially, the normality of the edge ring
as well as Gr\"obner bases of its toric ideal 
is extensively investigated.  However, the fundamental question
when an edge ring is Cohen--Macaulay is presumably open.    

Let $G$ be a finite simple graph,
i.e., a finite graph with no loop and no multiple edge, 
on the vertex set
$[d] = \{ 1, \ldots, d \}$ and
$E(G) = \{ e_1, \ldots, e_r \}$ its edge set.
Let $K[\tb] = K[t_1, \ldots, t_d]$ be the polynomial ring
in $d$ variables over a field $K$
and write $K[G]$ for the subring of $K[\tb]$ generated by those
squarefree quadratic monomials $\tb^e = t_it_j$
with $e = \{ i, j \} \in E(G)$.
The semigroup ring $K[G]$ is called the {\em edge ring} of $G$.  
Let $\K-d K[G]$ denote the Krull dimension of $K[G]$
and $\depth K[G]$ the depth of $K[G]$.
Let $K[\xb] = K[x_1, \ldots, x_r]$ be the polynomial ring
in $r$ variables over a field $K$.
The kernel $I_G$ of the surjective homomorphism 
$\pi : K[\xb] \to K[G]$ defined by setting
$\pi(x_i) = \tb^{e_i}$ for $i = 1, \ldots, r$
is called the {\em toric ideal} of $G$.  
One has $K[G] \cong K[\xb]/I_G$.
If $G$ is connected and is nonbipartite
(resp. bipartite), then $\K-d K[G] = d$ 
(resp. $\K-d K[G] = d - 1$).

The criterion of normality \cite[Corollary 2.3]{OH98} of edge rings
guarantees that $K[G]$ is normal if either $G$ is bipartite or $d \leq 6$.
If $d = 7$, then there exists a finite graph $G$
for which $K[G]$ is nonnormal.  However, 
it follows easily that $K[G]$ is Cohen--Macaulay whenever $d \leq 7$.
Computing the depth of the edge rings of 
all connected nonbipartite graphs $G$ with $7$ vertices
shows that the depth of $K[G]$ is at least $7$. 
Moreover, our computational experiment would naturally lead
the authors into the temptation to give the following 

\begin{conjecture}
\label{depthconjecture}
{\em
Let $G$ be a finite graph on $[d]$ with $d \geq 7$.
Then $\depth K[G] \geq 7$.
}
\end{conjecture}

Now, even though Conjecture \ref{depthconjecture} is completely open, 
by taking Conjecture \ref{depthconjecture} into consideration,
this paper will be devoted to proving the following 

\begin{theorem}
\label{main}
Given integers $f$ and $d$ with $7 \leq f \leq d$,
there exists a finite graph $G$ on $[d]$ with
$\depth K[G] = f$ and with $\K-d K[G] = d$.  
\end{theorem}   

Let $k \geq 1$ be an arbitrary integer and $G_{k+6}$ the finite graph 
on $[k+6]$ of Figure 0.1.
The essential part of a proof of Theorem \ref{main} is to show that
\begin{eqnarray}
\label{depth7}
\depth K[G_{k+6}] = \depth K[\xb]/I_{G_{k+6}} = 7.
\end{eqnarray}
In Section $1$, by virtue of the formula \cite[Theorem 2.1]{BCMP},  
the inequality $\depth K[G_{k+6}] \leq 7$ 
will be proved.
In Section $2$, we compute 
a Gr\"obner basis of $I_{G_{k+6}}$ and an initial ideal
$\initial(I_{G_{k+6}})$ of $I_{G_{k+6}}$, and show the inequality
$\depth K[\xb]/\initial(I_{G_{k+6}}) \geq 7$. 
In general, one has 
$\depth K[\xb]/I_{G_{k+6}} \geq \depth K[\xb]/\initial(I_{G_{k+6}})$
(e.g., \cite[Theorem 3.3.4 (d)]{HerzogHibi}).
Thus the desired equality (\ref{depth7}) follows.
 
\bigskip

\begin{center}
  \begin{picture}(300,120)(0,-10)
    \put(150,90){\circle*{5}}
    \put(150,70){\circle*{5}}
    \put(150,45){\circle*{1}}
    \put(150,40){\circle*{1}}
    \put(150,35){\circle*{1}}
    \put(150,10){\circle*{5}}
    \put(70,50){\circle*{5}}
    \put(20,90){\circle*{5}}
    \put(20,10){\circle*{5}}
    \put(230,50){\circle*{5}}
    \put(280,90){\circle*{5}}
    \put(280,10){\circle*{5}}
    \put(150,95){$7$}
    \put(150,75){$8$}
    \put(150,0){$k+6$}
    \put(70,55){$3$}
    \put(10,95){$1$}
    \put(10,0){$2$}
    \put(230,55){$4$}
    \put(285,95){$5$}
    \put(285,0){$6$}
    \thinlines
    \put(70,50){\line(2,1){80}}
    \put(70,50){\line(4,1){80}}
    \put(70,50){\line(2,-1){80}}
    \put(70,50){\line(-5,4){50}}
    \put(70,50){\line(-5,-4){50}}
    \put(230,50){\line(-2,1){80}}
    \put(230,50){\line(-4,1){80}}
    \put(230,50){\line(-2,-1){80}}
    \put(230,50){\line(5,4){50}}
    \put(230,50){\line(5,-4){50}}
    \put(20,90){\line(0,-1){80}}
    \put(280,90){\line(0,-1){80}}
    \put(100,75){$e_7$}
    \put(100,50){$e_9$}
    \put(100,10){$e_{2(k-1)+7}$}
    \put(45,75){$e_2$}
    \put(45,20){$e_3$}
    \put(185,75){$e_8$}
    \put(185,50){$e_{10}$}
    \put(185,10){$e_{2(k-1)+8}$}
    \put(245,75){$e_4$}
    \put(245,20){$e_5$}
    \put(5,45){$e_1$}
    \put(285,45){$e_6$}
  \end{picture}
\end{center}
\begin{center}
{\bf Figure 0.1.}
(finite graph $G_{k+6}$)
\end{center}

\bigskip
\bigskip

Once we know that $\depth K[G_{k+6}] = 7$,
to prove Theorem \ref{main} is straightforward.
In fact, given integers $f$ and $d$ with $7 \leq f \leq d$,
let $\Gamma$ denote the finite graph $G_{d - f + 7}$ on $[d - f + 7]$
and write $G$ for the finite graph on $[d]$ obtained from $\Gamma$
by adding $f - 7$ edges
\[
\{ 1, d - f + 8 \}, \{ 1, d - f + 9 \}, \ldots, \{ 1, d \}
\]
to $\Gamma$.
It then follows that 
$\depth K[G] = \depth K[\Gamma] + f - 7$. 
Since
$\depth K[\Gamma] = 7$, one has
$\depth K[G] = f$, as required.

\section{Proof of $\depth K[G_{k+6}] \leq 7$}
Let $G = G_{k+6}$ of Figure 0.1.
In this section, we prove that $\depth K[G] \leq 7$. 
Since the number of edges of $G$ is $r =  2(k-1) + 8$, 
Auslander--Buchsbaum formula implies that we may prove 
$\pd K[G] \geq r-7 = 2k-1$. 

\par
Let $S_G$ be the semigroup arising from $G$. 
Let ${\mathcal{A}}_G = \{ \underline{a}_1, \ldots, \underline{a}_r \}$ 
be the set of columns of the incidence matrix of $G$ where $\underline{a}_l$ 
corresponds to the edge $e_l$ (which corresponds to the variable $x_l$). 
Therefore, $S_G = \mathbb{N} {\mathcal{A}}_G$. 

\par
To prove $\pd K[G] \geq 2k-1$, we use the following theorem due to 
Briales, Campillo, Mariju\'{a}n, and Pis\'{o}n \cite{BCMP}. 
For $\underline{s} \in S_G$, we define the simplicial complex 
\begin{displaymath}
  {\Delta}_{\underline{s}} = \{ F \subset [r] \; : \; 
    \underline{s} - \underline{n}_F \in S_G \}, 
\end{displaymath}
where $\underline{n}_F = \sum_{l \in F} \underline{a}_l$. 
We denote by $\beta_{i,\underline{s}} (K[G])$, 
the $i$th multigraded Betti number of $K[G]$ in degree $\underline{s}$. 
\begin{lemma}[{\cite[Theorem 2.1]{BCMP}}]
  \label{claim:betti}
  Let $G$ be a finite simple graph. Then 
  \begin{displaymath}
    \beta_{j+1, \underline{s}} (K[G]) 
    = \dim_K \tilde{H}_j ({\Delta}_{\underline{s}}; K). 
  \end{displaymath}
\end{lemma}

\par
We consider the case where 
\begin{displaymath}
  \underline{s} = (1, 1, k+1, k+1, 1, 1, 2, 2, \ldots, 2). 
\end{displaymath}
By Lemma \ref{claim:betti}, it is sufficient to prove the following lemma: 
\begin{lemma}
  \label{claim:homology}
  Set $\underline{s} = (1, 1, k+1, k+1, 1, 1, 2, 2, \ldots, 2)$. Then 
  \begin{displaymath}
    \dim_K \tilde{H}_{2k-2} ({\Delta}_{\underline{s}}; K) \neq 0. 
  \end{displaymath}
\end{lemma}

\par
We set $\Delta = {\Delta}_{\underline{s}}$. 
Before proving Lemma \ref{claim:homology}, 
we compute the simplicial complex $\Delta$. 
\begin{lemma}
  Set $\underline{s} = (1, 1, k+1, k+1, 1, 1, 2, 2, \ldots, 2)$. 
  Then facets of ${\Delta}_{\underline{s}}$ are the following 
  subsets of $[r]$: 
  \begin{displaymath}
    \begin{alignedat}{3}
      F_{1,i} &= \{ 1, 4, 5, 7, 8, \ldots, 2(k-1)+8 \}
        \setminus \{ 2(i-1) + 8 \}, 
        &\qquad &i = 1, \ldots, k; \\
      F_{2,j} &= \{ 2, 3, 6, 7, 8, \ldots, 2(k-1)+8 \}
        \setminus \{ 2(j-1) + 7 \}, 
        &\qquad &j = 1, \ldots, k. \\
    \end{alignedat}
  \end{displaymath}
\end{lemma}
\begin{proof}
  Since $\underline{s} - \underline{n}_{F_{1,i}} 
  = \underline{a}_{2(i-1)+7} \in S_G$, we have 
  $F_{1,i} \in {\Delta}_{\underline{s}} = \Delta$. 
  (It follows that $\underline{s} \in S_G$.) 
  Similarly, we have $F_{2,j} \in \Delta$. 

  \par
  To prove that there are no facet other than $F_{1,i}, F_{2,j}$, 
  it is enough to show that 
  \begin{itemize}
  \item $\{ 1, 2 \}, \{ 1, 3 \}, \{ 4, 6 \}, \{ 5, 6 \} \notin \Delta$; 
  \item $\{ 1, 6 \} \notin \Delta$; 
  \item $\{ 2, 4 \}, \{ 2, 5 \}, \{ 3, 4 \}, \{ 3, 5 \} \notin \Delta$; 
  \item $F_0 = \{ 7, 8, \ldots, 2(k-1)+8 \} \notin \Delta$. 
  \end{itemize}

  \par
  Since the first entry of $\underline{s} - \underline{n}_{\{ 1, 2 \}}$ 
  is $-1 < 0$, it follows that 
  $\underline{s} - \underline{n}_{\{ 1, 2 \}} \notin S_G$. 
  Therefore $\{ 1, 2 \} \notin \Delta$. 
  By the symmetry, we also have  
  $\{ 1, 3 \}, \{ 4, 6 \}, \{ 5, 6 \} \notin \Delta$. 

  \par
  Second we show that $\{ 1, 6 \} \notin \Delta$. 
  Suppose, on the contrary, that $\{ 1, 6 \} \in \Delta$, i.e., 
  \begin{displaymath}
    \underline{s} - \underline{n}_{\{ 1, 6 \}} 
    = (0, 0, k+1, k+1, 0, 0, 2, 2, \ldots, 2) \in S_G.
  \end{displaymath}
  Then we can write $\underline{s} - \underline{n}_{\{ 1, 6 \}} 
  = \sum_{l=1}^{r} c_l \underline{a}_l$, where $c_l \in \mathbb{N}$. 
  Since $(\underline{s} - \underline{n}_{\{ 1, 6 \}})_1 
   = (\underline{s} - \underline{n}_{\{ 1, 6 \}})_2 = 0$ and 
  $(\underline{s} - \underline{n}_{\{ 1, 6 \}})_3 = k+1$, we have 
  $c_1 = c_2 = c_3 = 0$ and $\sum_{i=1}^k c_{2(i-1)+7} = k+1$. 
  Similarly, we have $c_4 = c_5 = c_6 = 0$ and 
  $\sum_{j=1}^k c_{2(j-1)+8} = k+1$. 
  Then $\sum_{i=1}^k c_{2(i-1)+7} + \sum_{j=1}^k c_{2(j-1)+8} = 2(k+1)$, 
  but it must be $2k$. This is a contradiction. 

  \par
  Next we show that 
  $\{ 2, 4 \}, \{ 2, 5 \}, \{ 3, 4 \}, \{ 3, 5 \} \notin \Delta$. 
  Suppose that $\{ 2, 4 \} \in \Delta$, i.e., 
  \begin{displaymath}
    \underline{s} - \underline{n}_{\{ 2, 4 \}} 
    = (0, 1, k, k, 0, 1, 2, 2, \ldots, 2) \in S_G.
  \end{displaymath}
  Then we can write $\underline{s} - \underline{n}_{\{ 2, 4 \}} 
  = \sum_{l=1}^{r} c_l \underline{a}_l$, where $c_l \in \mathbb{N}$. 
  Since $(\underline{s} - \underline{n}_{\{ 2, 4 \}})_1 = 0$ and 
  $(\underline{s} - \underline{n}_{\{ 2, 4 \}})_2 = 1$, we have $c_3 = 1$. 
  Similarly, we have $c_5 = 1$. Thus
  \begin{displaymath}
    (0, 0, k-1, k-1, 0, 0, 2, 2, \ldots, 2) \in S_G.
  \end{displaymath}
  Then the similar argument on the proof of $\{ 1, 6 \} \notin \Delta$ 
  yields a contradiction.  Therefore $\{ 2, 4 \} \notin \Delta$. 
  By the symmetry, we also have 
  $\{ 2, 5 \}, \{ 3, 4 \}, \{ 3, 5 \} \notin \Delta$. 

  \par
  Last, we show $F_0 \notin \Delta$. It follows from 
  \begin{displaymath}
    \underline{s} - \underline{n}_{F_0} 
    = (1, 1, 1, 1, 1, 1, 0, 0, \ldots, 0) \notin S_G.
  \end{displaymath}
\end{proof}

\par
Now we prove Lemma \ref{claim:homology}. 
\begin{proof}[Proof of Lemma \ref{claim:homology}]
Let $\Delta_1$ be the subcomplex of $\Delta$ 
whose facets are $F_{1,i}$, $i=1, \ldots, k$, and $\Delta_2$ 
the subcomplex of $\Delta$ whose facets are $F_{2,j}$, $j=1, \ldots, k$. 
Then $\Delta = \Delta_1 \cup \Delta_2$. Also facets of the simplicial complex 
$\Delta_1 \cap \Delta_2$ are 
\begin{displaymath}
  \{ 7, 8, \ldots, 2(k-1) + 8 \} \setminus \{ 2(j-1) + 7, 2(i-1) + 8 \}, 
  \qquad i, j = 1, \ldots, k. 
\end{displaymath}
In particular, $\dim \Delta_1 \cap \Delta_2 = 2k-3$. 
Note that both of $\Delta_1$ and $\Delta_2$ are cones over some simplicial 
complexes and so the reduced homologies of these all vanish. Therefore the
Mayer--Vietoris sequence 
\begin{displaymath}
  \begin{aligned}
    \cdots 
    &\longrightarrow \tilde{H}_{i} (\Delta_1 \cap \Delta_2; K)
    \longrightarrow \tilde{H}_i (\Delta_1; K) \oplus \tilde{H}_i (\Delta_2; K)
    \longrightarrow \tilde{H}_i (\Delta; K) \\
    &\longrightarrow \tilde{H}_{i-1} (\Delta_1 \cap \Delta_2; K)
    \longrightarrow 
      \tilde{H}_{i-1} (\Delta_1; K) \oplus \tilde{H}_{i-1} (\Delta_2; K)
    \longrightarrow \cdots
  \end{aligned}
\end{displaymath}
yields 
\begin{displaymath}
  \tilde{H}_i (\Delta; K) \cong \tilde{H}_{i-1} (\Delta_1 \cap \Delta_2; K)
  \qquad \text{for all $i$.}
\end{displaymath}

\par
We can see $\tilde{H}_{2k-3} (\Delta_1 \cap \Delta_2; K) \neq 0$ 
by considering the alternating sum of all facets of $\Delta_1 \cap \Delta_2$: 
\begin{displaymath}
  \sum_{1 \leq i,j \leq k} (-1)^{i+j} 
  \{ 7, 8, \ldots, 2(k-1) + 8 \} \setminus \{ 2(j-1) + 7, 2(i-1) + 8 \}. 
\end{displaymath}
Therefore we have $\tilde{H}_{2k-2} (\Delta; K) \neq 0$. 

\end{proof}

\section{Proof of $\depth K[G_{k+6}] \geq 7$}
Let, as before, $G = G_{k+6}$ of Figure 0.1.
In this section we prove that $\depth K[G] \geq 7$. 

\par
We set $C_1 = (e_2, e_1, e_3)$ and $C_2 = (e_4, e_6, e_5)$, 
both of which are $3$-cycles of $G$. 
By \cite[Lemma 3.2]{OH99}, 
there are $3$ kinds of primitive even closed walks $\Gamma$ of $G$ 
up to the way: 
\begin{enumerate}
\item[(I)] a $4$-cycle: 
  $\Gamma = (e_{2(i-1)+7}, e_{2(i-1)+8}, e_{2(j-1)+8}, e_{2(j-1)+7})$, 
  where $i<j$; 
\item[(II)] a walk on two $3$-cycles $C_1, C_2$ and a single path 
  connecting $C_1$ and $C_2$: 
  $\Gamma 
    = (C_1, e_{2(i-1)+7}, e_{2(i-1)+8}, C_2, e_{2(i-1)+8}, e_{2(i-1)+7})$, 
  where $i=1, \ldots, k$; 
\item[(III)] a walk on two $3$-cycles $C_1, C_2$ and two different paths 
  combining $C_1$ and $C_2$: 
  $\Gamma 
    = (C_1, e_{2(i-1)+7}, e_{2(i-1)+8}, C_2, e_{2(j-1)+8}, e_{2(j-1)+7})$, 
  where $i<j$. 
\end{enumerate}

It was proved in \cite[Lemma 3.1]{OH99} that binomials corresponding 
to these primitive even closed walks generate the toric ideal $I_G$. 
Let us consider the lexicographic order $< = <_{\lex}$ with 
$x_1 > x_2 > x_3 > \cdots > x_{2(k-1)+8}$. 

\begin{lemma}
  The set of binomials corresponding to primitive even closed walks 
  (I), (II), (III) is a Gr\"{o}bner basis of $I_G$ with respect to $<_{\lex}$. 
\end{lemma}
\begin{proof}
The result follows from a straightforward application of Buchberger's algorithm to the set of generators of $I_G$ corresponding to the primitive even closed walks listed above.  Let $f$ and $g$ be two such generators. 
We will prove that the $S$-polynomial, $S(f,g)$, yielding from 
Buchberger's algorithm will reduce to $0$ by generators of type 
(I), (II) and (III). 
For convenience of notation, we will assume that $i,j,p,$ and $q$ are all odd integers such that $7\leq i<j, 7\leq p<q$.

\underline{Case 1:} Let $f = x_ix_{j+1}-x_{i+1}x_j$ and $g = x_px_{q+1}-x_{p+1}x_q$ be generators of type (I).  If $i\neq p$ and $j\neq q$, then the leading terms of $f$ and $g$ are relatively prime and thus the S-polynomial $S(f,g)$ will reduce to $0$
(e.g., \cite[Lemma 2.3.1]{HerzogHibi}).  Suppose $i = p$, then 
\begin{eqnarray*}
S(f,g) &=& \frac{\lcm(f,g)}{LT_{<_{\lex}}(f)}f-\frac{\lcm(f,g)}{LT_{<_{\lex}}(g)}g\\
	     &=& x_{q+1}(x_ix_{j+1}-x_{i+1}x_j)-x_{j+1}(x_ix_{q+1}-x_{i+1}x_q)\\
	     &=& x_{i+1}x_{j+1}x_q-x_{i+1}x_jx_{q+1}\\
	     &=& x_{i+1}(x_{j+1}x_q-x_jx_{q+1}).
\end{eqnarray*}
Note that, up to sign, $x_{j+1}x_q-x_jx_{q+1}$ is a generator of $I_G$ of type (I) and therefore $S(f,g)$ will reduce to $0$.  The case of $j=q$ is similar.  

\underline{Case 2:} Let $f$ be the same as above and $g = x_1x_4x_5x_p^2-x_2x_3x_6x_{p+1}^2$ a generator of type (II).  
If $i \neq p$ then the leading terms of $f$ and $g$ are relatively prime 
and therefore negligible.  If $i = p$ then
\begin{eqnarray*}
S(f,g) &=& x_1x_4x_5x_i(x_ix_{j+1}-x_{i+1}x_j)-x_{j+1}(x_1x_4x_5x_i^2-x_2x_3x_6x_{i+1}^2)\\
		 &=& x_2x_3x_6x_{i+1}^2x_{j+1}-x_1x_4x_5x_ix_{i+1}x_j\\
	     &=& -x_{i+1}(x_1x_4x_5x_ix_j-x_2x_3x_6x_{i+1}x_{j+1})
\end{eqnarray*}
where $x_1x_4x_5x_ix_j-x_2x_3x_6x_{i+1}x_{j+1}$ is a generator of type (III). 

\underline{Case 3:}  Again, we assume that $f$ is the same as above.  Now assume $g$ is of type (III), $g=x_1x_4x_5x_px_q-x_2x_3x_6x_{p+1}x_{q+1}$.  If $i\neq p,q$ then the leading terms of $f$ and $g$ will be relatively prime.  Suppose $i = p$, then
\begin{eqnarray*}
S(f,g) &=& x_1x_4x_5x_q(x_ix_{j+1}-x_{i+1}x_j)-x_{j+1}(x_1x_4x_5x_ix_q-x_2x_3x_6x_{i+1}x_{q+1})\\
	     &=& -x_{i+1}(x_1x_4x_5x_qx_j-x_2x_3x_6x_{q+1}x_{j+1})
\end{eqnarray*}
and again we have that $x_1x_4x_5x_qx_j-x_2x_3x_6x_{q+1}x_{j+1}$ 
is either a type (II) or type (III) generator of $I_G$.  
The case of $i=q$ is similar.

\underline{Case 4:}  Now let $f$ and $g$ both be generators of type (II), 
$f=x_1x_4x_5x_i^2-x_2x_3x_6x_{i+1}^2, g=x_1x_4x_5x_j^2-x_2x_3x_6x_{j+1}^2.$  Then the $S$-polynomial
\begin{eqnarray*}
S(f,g) &=& x_j^2(x_1x_4x_5x_i^2-x_2x_3x_6x_{i+1}^2)-x_i^2(x_1x_4x_5x_j^2-x_2x_3x_6x_{j+1}^2)\\
	     &=& x_2x_3x_6(x_i^2x_{j+1}^2-x_{i+1}^2x_j^2)\\
	     &=& x_2x_3x_6(x_ix_{j+1}+x_{i+1}x_j)(x_ix_{j+1}-x_{i+1}x_j)
\end{eqnarray*}
is a multiple of a type (I) generator.

\underline{Case 5:}  Let $f$ be the same as in Case 4 and $g=x_1x_4x_5x_px_q-x_2x_3x_6x_{p+1}x_{q+1}$ of type (III).  First suppose that $i\neq p,q$. 
Let us consider the case of $i<p$. Then 
\begin{eqnarray*}
S(f,g) &=& x_px_q(x_1x_4x_5x_i^2-x_2x_3x_6x_{i+1}^2)-x_i^2(x_1x_4x_5x_px_q-x_2x_3x_6x_{p+1}x_{q+1})\\
	     &=& x_2x_3x_6(x_i^2x_{p+1}x_{q+1}-x_{i+1}^2x_px_q)\\
	     &=& x_2x_3x_6[x_ix_{q+1}(x_ix_{p+1}-x_{i+1}x_p)+x_ix_{i+1}x_px_{q+1}-x_{i+1}^2x_px_q]\\
	     &=& x_2x_3x_6[x_ix_{q+1}(x_ix_{p+1}-x_{i+1}x_p)+x_{i+1}x_p(x_ix_{q+1}-x_{i+1}x_q)].
\end{eqnarray*}
And so $S(f,g)$ reduce to 0 by two type (I) generators.
The cases of $p<i<q$ and $q<i$ are similar. 

Now suppose $i = p$, then  the $S$-polynomial,
\begin{eqnarray*}
S(f,g) &=& x_q(x_1x_4x_5x_i^2-x_2x_3x_6x_{i+1}^2)-x_i(x_1x_4x_5x_ix_q-x_2x_3x_6x_{i+1}x_{q+1})\\
	     &=& x_2x_3x_6x_{i+1}(x_{i}x_{q+1}-x_{i+1}x_q).
\end{eqnarray*}
is a multiple of a type (I) generator.  The case of $i = q$ is similar.

\underline{Case 6:}  Finally, we let consider the case that both $f$ and $g$ are of type (III): $f = x_1x_4x_5x_ix_j-x_2x_3x_6x_{i+1}x_{j+1}$, $g =x_1x_4x_5x_px_q-x_2x_3x_6x_{p+1}x_{q+1}$. 
We may assume that $i \leq p$. 
Let us first suppose that $i,j\neq p,q$, then
\begin{eqnarray*}
S(f,g) &=& x_px_q(x_1x_4x_5x_ix_j-x_2x_3x_6x_{i+1}x_{j+1})-x_ix_j(x_1x_4x_5x_px_q-x_2x_3x_6x_{p+1}x_{q+1})\\
         &=& x_2x_3x_6(x_ix_jx_{p+1}x_{q+1}-x_{i+1}x_{j+1}x_px_q)\\
         &=& x_2x_3x_6[x_jx_{q+1}(x_ix_{p+1}-x_{i+1}x_p) + x_{i+1}x_p(x_jx_{q+1} -x_{j+1}x_q)].
\end{eqnarray*}

Now let $i =p$.  We then have
\begin{eqnarray*}
S(f,g) &=& x_qf-x_jg = -x_qx_2x_3x_6x_{i+1}x_{j+1} + x_jx_2x_3x_6x_{i+1}x_{q+1}\\
         &=& x_2x_3x_6x_{i+1}(x_jx_{q+1}-x_{j+1}x_q).
\end{eqnarray*}
The cases of $j=p$ and $j=q$ are similar.

\end{proof}

\par
Now we prove that $\depth K[G] \geq 7$. We denote by $\initial (I_G)$, 
the initial ideal of $I_G$ with respect to $<_{\lex}$. 
Since 
\begin{displaymath}
  \depth K[G] = \depth K[\xb]/I_G \geq \depth K[\xb]/\initial (I_G), 
\end{displaymath}
it is sufficient to prove that $\depth K[\xb]/\initial (I_G) \geq 7$. 
By Auslander--Buchsbaum formula, it is enough to prove the following lemma: 
\begin{lemma}
  \begin{displaymath}
    \pd_{K[\xb]} K[\xb]/\initial (I_G) \leq 2k-1. 
  \end{displaymath}
\end{lemma}
\begin{proof}
  First we compute $\initial (I_G)$. 

  \par
  The binomials corresponding to type (I) are 
  \begin{displaymath}
    x_{2(i-1)+7} x_{2(j-1)+8} - x_{2(i-1)+8} x_{2(j-1)+7}, 
    \qquad \text{where $i<j$}. 
  \end{displaymath}
  The initial term of this binomial is $x_{2(i-1)+7} x_{2(j-1)+8}$ ($i<j$). 
  We denote by $I'$, the ideal generated by these monomials. 
  Note that $x_8$ and $x_{2(k-1)+7}$ do not appear in the minimal system 
  of monomial generators of $I'$. 

  \par
  The binomials corresponding to types (II), (III) are
  \begin{displaymath}
    x_2 x_3 x_6 x_{2(i-1)+8} x_{2(j-1)+8} 
    - x_1 x_4 x_5 x_{2(i-1)+7} x_{2(j-1)+7}, 
    \qquad \text{where $i \leq j$}. 
  \end{displaymath}
  The initial term of this binomial is 
  $- x_1 x_4 x_5 x_{2(i-1)+7} x_{2(j-1)+7}$ ($i \leq j$). 

  \par
  Therefore 
  \begin{displaymath}
    \begin{aligned}
      \initial (I_G) 
      &= x_1 x_4 x_5 (x_7, x_9, \ldots, x_{2(k-1)+7})^2 +I' \\
      &= ((x_7, x_9, \ldots, x_{2(k-1)+7})^2 +I') \cap ((x_1 x_4 x_5) + I'). 
    \end{aligned}
  \end{displaymath}
  We set 
  \begin{displaymath}
    \begin{aligned}
      I_1 &= (x_7, x_9, \ldots, x_{2(k-1)+7})^2 +I' \\
      I_2 &= (x_1 x_4 x_5) + I'. 
    \end{aligned}
  \end{displaymath}
  By the short exact sequence 
  $0 \rightarrow K[\xb]/I_1 \cap I_2 \rightarrow K[\xb]/I_1 \oplus K[\xb]/I_2 
     \rightarrow K[\xb]/(I_1 + I_2) \rightarrow 0$, 
  we have 
  \begin{equation}
    \label{eq:pdineq}
    \pd_{K[\xb]} K[\xb]/\initial (I_G)
    \leq \max \{ \pd_{K[\xb]} K[\xb]/I_1, \pd_{K[\xb]} K[\xb]/I_2, \pd_{K[\xb]} K[\xb]/(I_1 + I_2) - 1 \}. 
  \end{equation}

  \par
  Now we investigate each of $\pd_{K[\xb]} K[\xb]/I_1, \pd_{K[\xb]} K[\xb]/I_2, \pd_{K[\xb]} K[\xb]/(I_1 + I_2)$. 

  \par
  First we consider the ideal $I_1$. Note that $x_1, \ldots, x_6$ and $x_8$ 
  do not appear in the minimal system of monomial generators of $I_1$. 
  Let $K[{\xb}']$ be the polynomial ring over $K$ with variables 
  $x_7, x_9, x_{10}, \ldots, x_{2(k-1)+8}$. 
  Then 
  $\pd_{K[\xb]} K[\xb]/I_1 = \pd_{K[{\xb}']} K[{\xb}']/(I_1 \cap K[{\xb}'])$. 
  By Hilbert's syzygy theorem, we have 
  $\pd_{K[{\xb}']} K[{\xb}']/(I_1 \cap K[{\xb}']) \leq 2k-1$. 

  \par
  Next we consider the ideal $I_2 = (x_1 x_4 x_5) + I'$. 
  Since the variables $x_1, x_4, x_5$ do not appear in the minimal systems of 
  generators of $I'$, we have 
  \begin{displaymath}
    \pd_{K[\xb]} K[\xb]/I_2 = \pd_{K[\xb]} K[\xb]/I' + \pd_{K[\xb]} K[\xb]/(x_1 x_4 x_5) = \pd_{K[\xb]} K[\xb]/I' + 1. 
  \end{displaymath}
  Then similarly to the case of $I_1$, we have $\pd_{K[\xb]} K[\xb]/I' \leq 2k-2$. 
  Thus we have $\pd_{K[\xb]} K[\xb]/I_2 \leq 2k-1$. 

  \par
  Last, we consider the ideal $I_1 + I_2 = (x_1 x_4 x_5) + I_1$. 
  The same reason as the case of $I_2$, we have 
  $\pd_{K[\xb]} K[\xb]/(I_1 + I_2) = \pd_{K[\xb]} K[\xb]/I_1 + 1 \leq 2k$. 

  \par
  Combining these results with (\ref{eq:pdineq}), we have 
  $\pd_{K[\xb]} K[\xb]/\initial (I_G) \leq 2k-1$, as desired. 
\end{proof}


\end{document}